\documentclass[a4paper,reqno]{amsart}
\usepackage{mathptmx} 

\usepackage{amssymb}        
\usepackage{amsmath}        

\providecommand{\cal}[1]{\mathcal{#1}}

\newcommand{\JJgen}{\mathbf{A}}

\newcommand{\JJB}{{\mathbb{B}}}
\newcommand{\JJC}{{\mathbb{C}}}
\newcommand{\JJN}{{\mathbb{N}}}
\newcommand{\JJR}{{\mathbb{R}}}
\newcommand{\JJRn}{{\mathbb{R}}^{n}}

\newcommand{\JJlap}{\operatorname{\Delta}}
\newcommand{\JJmlap}{-\!\operatorname{\Delta}}
\newcommand{\JJscal}[2]{(\,#1\,|\, #2\,)}
\newcommand{\JJSet}[2]{\bigl\{\,#1\bigm| #2\,\bigr\}}
\newcommand{\JJvvvert}{{|\hspace{-1.6pt}|\hspace{-1.6pt}|}}
\newcommand{\JJdual}[2]{\ensuremath{\langle{#1},{#2}\rangle}} 

\def\Xint#1{\mathchoice
{\XXint\displaystyle\textstyle{#1}}%
{\XXint\textstyle\scriptstyle{#1}}%
{\XXint\scriptstyle\scriptscriptstyle{#1}}%
{\XXint\scriptscriptstyle\scriptscriptstyle{#1}}%
\!\int}
\def\XXint#1#2#3{{\setbox0=\hbox{$#1{#2#3}{\int}$}
\vcenter{\hbox{$#2#3$}}\kern-.5\wd0}}
\def\JJdashint{\Xint-}

\newtheorem{theorem}{Theorem}
\newtheorem{proposition}{Proposition}
\newtheorem{lemma}{Lemma}

\theoremstyle{definition}

\theoremstyle{remark}
\newtheorem{remark}{Remark}

\begin{document}
\title{A Class of Well-Posed Parabolic Final Value Problems}

\author{Jon Johnsen}
\address{Department of Mathematics, Aalborg University, Skjernvej 7A, DK-9220 Denmark}
\email{jjohnsen@math.aau.dk}

%
%

\begin{abstract}{This paper focuses on parabolic final
  value problems, and well-posedness is proved for a large class of these.
  The clarification is obtained from Hilbert spaces that characterise data
  that give existence, uniqueness and stability of the solutions.  
  The data space is the graph normed domain of an unbounded operator that maps final states
  to the corresponding initial states. It induces a new compatibility condition, 
  depending crucially on the fact that analytic semigroups always are invertible in the class of closed operators. 
  Lax--Milgram operators in vector distribution spaces are the main framework.
  The final value heat conduction problem on a smooth open set is
  also proved to be well posed, and non-zero Dirichlet data are shown to require an extended
  compatibility condition obtained by adding an improper Bochner integral.  
}
\end{abstract}


\maketitle

\section{Introduction}

Well-posedness of final value problems
for a large class of parabolic differential equations was recently obtained by the author jointly
with A.\ E.\ Christensen, and an ample description was given for a broad audience in \cite{JJChJo18ax},
after the announcement in \cite{JJChJo18}. The present exposition is more
concise and incorporates some improvements that, as indicated, may lead to
future developments of the theory. 

The below theoretical analysis of the problems shows that they are well posed, i.e., they  have 
\emph{existence, uniqueness} and \emph{stability} of solutions $u\in X$ for given data, say
$(f,u_T)\in Y$, in~certain Hilbertable spaces $X$, $Y$ that are described explicitly.

This has seemingly closed a gap in the theory, which has remained since the 1950's,
although the well-posedness is decisive for the interpretation and accuracy of 
numerical schemes that would be used in practice (John~\cite{JJJohn55} made an early
contribution in this direction).
In rough terms, the results are derived from a useful structure on the reachable set for a
general class of parabolic differential equations.

The primary example (addressed in the end of the paper) is the heat conduction problem of characterising the functions
$u(t,x)$ that in a $C^\infty$-smooth bounded open set 
$\Omega\subset\JJRn$ with boundary $\partial\Omega$
fulfil the equations (whereby $\Delta=\partial_{x_1}^2+\dots+\partial_{x_n}^2$),
\begin{equation}  \label{JJheat-intro}
\left\{
\begin{aligned}
  \partial_tu(t,x)-\JJlap u(t,x)&=f(t,x) &&\quad\text{for $t\in\,]0,T[\,$,  $x\in\Omega$,}
\\
   u(t,x)&=g(t,x) &&\quad\text{for $t\in\,]0,T[\,$, $x\in\partial\Omega$,}
\\
  u(T,x)&=u_T(x) &&\quad\text{for $x\in\Omega$}.
\end{aligned}
\right.
\end{equation}
An area of interest of this could be a nuclear power plant 
hit by a power failure at time $t=0$: after power is regained at $t=T>0$, and 
the reactor temperatures $u_T(x)$ are measured, it would be crucial to
calculate backwards to settle whether at an earlier time $t_0<T$ the temperatures $u(t_0,x)$
could cause damage to the fuel rods.

A short plan of the paper is the following: the abstract result on final value problems is given in
Section~\ref{JJafvp-ssect} below. Its proof then follows in Section~\ref{JJaproof-sect}. The results and
proofs for the heat problem in \eqref{JJheat-intro} are presented
Theorems~\ref{JJyz-thm}--\ref{JJyz-thm''} in Section~\ref{JJheat-sect}.

\subsection{Background: Phenomena of Instability}
It may be recalled that there is a phenomenon of $L_2$-instability in the homogeneous case
$f=0$, $g=0$ in \eqref{JJheat-intro}. This classical fact was perhaps first described in 1961 by Miranker 
\cite{JJMiranker61}; but it was also emphasized more recently by Isakov~\cite{JJIsa98}.

The instability results by considering the Dirichlet realization
of the Laplace operator, written $\JJmlap_D$, and the  $L_2(\Omega)$-orthonormal
basis $e_1(x), e_2(x),\dots$ of eigenfunctions associated to the
usual ordering of its eigenvalues
$0<\lambda_1\le\lambda_2\le\dots$, which via Weyl's law for the counting function, cf.\ \cite[Ch.~6.4]{JJCuHi53}, gives
\begin{equation} \label{JJWeyl-id}
  \lambda_j={\cal O}(j^{2/n})\quad\text{ for $j\to\infty$}.
\end{equation}
The orthonormal basis gives rise to a sequence of final value data,
\begin{equation}
  u_{T,j}(x)=e_j(x)\quad\text{ for $j\in\JJN$}.
\end{equation}
This sequence clearly lies on the unit sphere in $L_2(\Omega)$ as
$\|u_{T,j}\|= \|e_j\|=1$ for $j\in\JJN$.
But the corresponding solutions $u_j$ to the heat equation $u'\JJmlap u=0$, namely
\begin{equation}
  u_j(t,x)= e^{(T-t)\lambda_j}e_j(x),
\end{equation}
obviously have initial states $u(0,x)$ with $L_2$ norms that because of \eqref{JJWeyl-id} 
grow \emph{rapidly} with the index $j$,
\begin{equation}
  \|u_j(0,\cdot)\| = e^{T\lambda_j}\|e_j\| = e^{T\lambda_j}\nearrow\infty.
\end{equation}

Consequently, when a final state $u_T(x)$ is approximated through measurements $\tilde
u_{T,1}(x)$, $\tilde u_{T,2}(x)$, \dots\ made with increasing accuracy in $L_2(\Omega)$-norm, it is likely that each
difference $\tilde u_{T,j+1}-\tilde u_{T,j}$ has more non-zero coordinates with respect to the
orthonormal basis $(e_n)_{n\in\JJN}$ than the previous one, which in view of the above rapid
blow-up makes it likely that the calculated initial state correction 
$\tilde u_{j+1}(0,x)-\tilde u_j(0,x)$ will be of the same size (in $L_2$-norm) as the previous one, $\tilde u_{j}(0,x)-\tilde u_{j-1}(0,x)$.
In this case, it remains entirely unclear whether or not $\tilde u_{j+1}(0,x)$ is a more accurate
approximation to $u(0,x)$ than $\tilde u_j(0,x)$---or if $\tilde u_{j+1}(x)$ has been calculated in vein.

The above $L_2$-instability cannot be explained away, of course, but in reality it
only indicates that the $L_2$-norm is an insensitive choice for problem \eqref{JJheat-intro}. It
therefore seems reasonable to pose the rethorical
\begin{quote}
  {\bf Question:} Is the final value heat problem in \eqref{JJheat-intro} well posed?
\end{quote}
While the answer is ``yes'', one could wonder why this has not been proved before.
Here it should be mentioned that our description of reachable 
sets for parabolic problems exploits the previously unavailable structures in the next section.

\subsection{Main Tool: Injectivity}
The key to the analysis of final value problems lies at a rather different spot, namely, that an
analytic semigroup of operators (like $e^{t\JJlap_D}$) always consists of \emph{injections}. This enters 
both at the technical and  conceptual level, that is, injectivity enters not just in the proofs, 
but also in the objects that the theorems are concerned with.

A few aspects of semigroup theory in a complex Banach space $B$ is therefore recalled here.
Besides classical references by Davies~\cite{JJDav80}, Pazy~\cite{JJPaz83}, Tanabe~\cite{JJTan79} or Yosida~\cite{JJYos80},
a more recent account is given in \cite{JJABHN11}.

The generator is
$\JJgen x=\lim_{t\to0^+}\frac1t(e^{t\JJgen}x-x)$, where $x$ belongs to the domain  $D(\JJgen)$ when the
limit exists. $\JJgen$ is a densely defined, closed linear 
operator in $B$ that for some $\omega \geq 0$, $M \geq 1$ satisfies 
$\|(\JJgen-\lambda)^{-n}\|_{\JJB(B)}\le M/(\lambda-\omega)^n$ for $\lambda>\omega$, $n\in\JJN$.

The corresponding $C_0$-semigroup of operators $e^{t\JJgen}\in\JJB(B)$ is of type $(M,\omega)$: 
it fulfils that $e^{t\JJgen}e^{s \JJgen}=e^{(s+t)\JJgen}$ for $s,t\ge0$, $e^{0\JJgen}=I$,
$\lim_{t\to0^+}e^{t \JJgen}x=x$ for $x\in B$, and
\begin{align}  
  \|e^{t\JJgen}\|_{\JJB(B)} \leq M e^{\omega t} \quad \text{ for } 0 \leq t < \infty.
\end{align}
Indeed, the Laplace transformation 
$(\lambda I-\JJgen)^{-1}=\int_0^\infty e^{-t\lambda}e^{t\JJgen}\,dt$ 
gives a bijection of the semigroups of type $(M,\omega)$ onto (the resolvents of) the stated class of generators.

The well-known result below gives a criterion for $\JJgen$ to generate a
$C_0$-semigroup $e^{z\JJgen}$ that is defined and analytic for $z$ in the open sector
\begin{equation}
  S_{\theta}= \JJSet{z\in\JJC}{z\ne0,\ |\arg z | < \theta}.
\end{equation}
It is formulated  in terms of the spectral sector
\begin{equation} 
  \Sigma_{\theta} 
  =\JJSet{ \lambda \in\JJC}{ |\arg\lambda | < \frac{\pi}{2} + \theta} \cup \left\{ 0 \right\}.
\end{equation}

\begin{proposition}  \label{JJPazy'-prop} 
When $\JJgen$ generates a $C_0$-semigroup of type $(M,\omega)$ and $\omega\in\rho(\JJgen)$,
then the following properties are equivalent for 
$\theta \in\,]0,\frac{\pi}{2}[\,$:
\begin{itemize}
  \item[{\rm (i)}]
  The resolvent set $\rho(\JJgen)$ fulfils $\rho(\JJgen) \supset \omega+\Sigma_{\theta}$  and 
\begin{equation} 
  \sup\JJSet{|\lambda-\omega|\cdot\|(\lambda I - \JJgen)^{-1} \|_{\JJB(B)}}{\lambda\in\omega+\Sigma_{\theta}, \
    \lambda \neq \omega} <\infty. 
\end{equation}
  \item[{\rm (ii)}] 
 The semigroup $e^{t \JJgen}$ extends to an analytic semigroup $e^{z \JJgen}$ defined for $z\in
 S_{\theta}$ with
\begin{equation}
   \sup\JJSet{ e^{-z\omega}\|e^{z\JJgen}\|_{\JJB(B)}}{z\in \overline{S}_{\theta'}}<\infty \quad \text{whenever $0<\theta'<\theta$}. 
\end{equation}
\end{itemize}
In the affirmative case, 
$e^{t \JJgen}$ is differentiable in $\JJB(B)$ for $t>0$, $(e^{t\JJgen})' = \JJgen e^{t\JJgen}$ and
\begin{align} 
   \sup_{t>0} te^{-t\omega}\|\JJgen e^{t\JJgen}\|_{\JJB(B)} <\infty .
\end{align}
\end{proposition}

In case $\omega=0$,
this is just a review of the main parts of Theorem~2.5.2 in \cite{JJPaz83}.
For general $\omega\ge0$, one can reduce to this case, since $\JJgen=\omega I+(\JJgen-\omega I)$ yields the operator
identity $e^{t\JJgen}=e^{t\omega}e^{t(\JJgen-\omega I)}$, where $e^{t(\JJgen-\omega I)}$ is of type
$(M,0)$. Indeed, the right-hand side is easily seen to be a $C_0$-semigroup, which since $e^{t\omega}=1+t\omega+o(t)$ has
$\JJgen$ as its generator, so the identity results from the bijectiveness of the Laplace transform.
In this way, (i)$\iff$(ii) follows straightforwardly from the case $\omega=0$, using for both implications that 
$e^{z\JJgen}=e^{z\omega}e^{z(\JJgen-\omega I)}$ holds in $S_\theta$ by unique analytic extension.

To elucidate the role of \emph{injectivity} of $e^{t \JJgen}$, 
it is recalled that if $e^{t\JJgen}$ is analytic, then $u'=\JJgen u$, $u(0)=u_0$ is always uniquely solved by
$u(t)=e^{t\JJgen}u_0$ for \emph{every} $u_0\in B$. 
Here injectivity clearly is  equivalent to the important geometric property that the trajectories of two solutions
$e^{t\JJgen}v_0$ and $e^{t\JJgen}w_0$ of $u'=\JJgen u$  have no confluence point in $B$ for $v_0\ne w_0$.
Nevertheless, the literature seems to have focused only on examples of semigroups with non-invertibility of
$e^{t\JJgen}$, like \cite[Ex.~2.2.1]{JJPaz83}.

The reason for stating Proposition~\ref{JJPazy'-prop} for general type $(M,\omega)$
semigroups is that it shows explicitly that cases with $\omega>0$ only have different estimates in
the closed subsectors $\overline{S}_{\theta'}$---but the mere analyticity in $S_{\theta}$ 
is unaffected by the translation by $\omega I$. One therefore has the following improved version of
\cite[Prop.\ 1]{JJChJo18ax}:

\begin{proposition}  \label{JJinj-prop}
If a $C_0$-semigroup $e^{t\JJgen}$ of type $(M,\omega)$ on a complex Banach space $B$ has an
analytic extension $e^{z\JJgen}$ to
$S_{\theta}$ for $\theta>0$, then $e^{z\JJgen}$ is \emph{injective} for every $z \in S_\theta$.
\end{proposition}

\begin{proof}
Let $e^{z_0 \JJgen} u_0 = 0$ hold for some $u_0 \in B$ and $z_0 \in S$.
Analyticity of $e^{z\JJgen}$ in $S_{\theta}$ carries over by the differential calculus in Banach
spaces to the map $f(z)= e^{z\JJgen}u_0$.
So for $z$ in a suitable open ball $B(z_0,r)\subset S_{\theta}$, a Taylor expansion and the identity
$f^{(n)}(z_0) = \JJgen^n e^{z_0 \JJgen}u_0$ for analytic semigroups (cf.~\cite[Lem.~2.4.2]{JJPaz83})  give
\begin{align}
  f(z) = \sum_{n=0}^{\infty} \frac{1}{n!}(z-z_0)^n f^{(n)}(z_0)=\sum_{n=0}^{\infty}
  \frac{1}{n!}(z-z_0)^n 
    \JJgen^n e^{z_0 \JJgen}u_0\equiv 0.
\end{align}
Hence $f\equiv 0$ on $S_{\theta}$ by unique analytic extension.
Now, as $e^{t\JJgen}$ is strongly continuous,
$u_0 = \lim_{t \rightarrow 0^+} e^{t\JJgen} u_0 = \lim_{t \rightarrow 0^+}f(t) = 0$.
Thus the null space of $e^{z_0\JJgen}$ is trivial.
\end{proof}

\begin{remark} 
The injectivity in Proposition~\ref{JJinj-prop} 
was claimed by Showalter \cite{JJSho74} for  
$z>0$, $\theta\le \pi/4$ and $B$ a Hilbert space
(with a flawed proof, as noted in \cite[Rem.~1]{JJChJo18ax}, cf.\  details on the
counter-example in Lemma 3.1 and Remark 3 in \cite{JJ18logconv}). A local version for the
Laplacian on $\JJRn$ was given by Rauch \cite[Cor.~4.3.9]{JJRau91}.
\end{remark}

As a consequence of the above injectivity, for an \emph{analytic} semigroup 
$e^{t\JJgen}$ we may consider its inverse that, consistently with the case in which $e^{t\JJgen}$ forms a group in
$\mathbb{B}(B)$, may be denoted for $t>0$ by
$e^{-t\JJgen} = (e^{t\JJgen})^{-1}$.

Clearly $e^{-t\JJgen}$ maps its domain $D(e^{-t\JJgen})$, which is the range $R(e^{t\JJgen})$, bijectively onto $H$, and
it is in general an unbounded, but closed operator in $B$. 

Specialising to a Hilbert space $B=H$, then also $(e^{t\JJgen})^*=e^{t\JJgen^*}$ is analytic, so its null space 
$Z(e^{t\JJgen^*})=\{0\}$ by Proposition~\ref{JJinj-prop}, whence
$D(e^{-t\JJgen})$ is dense in $H$.

A partial group phenomenon and commutation properties are restated here:

\begin{proposition}{\cite[Prop.\;2]{JJChJo18ax}}  \label{JJinverse-prop}
The above inverses $e^{-t\JJgen}$ form a semigroup of unbounded operators in $H$,
\begin{equation} 
  e^{-s\JJgen}e^{-t\JJgen}= e^{-(s+t)\JJgen} \qquad \text{for $t, s\ge0$}.
\end{equation}
This extends to $(s,t)\in\JJR\times \,]-\infty,0]$, where $e^{-(t+s)\JJgen}$  may be unbounded for $t+s>0$. 
Moreover, as unbounded operators the $e^{-t\JJgen}$ commute with $e^{s \JJgen}\in \JJB(H)$, i.e.,
\begin{equation}
  e^{s \JJgen}e^{-t\JJgen}\subset e^{-t\JJgen}e^{s\JJgen} \qquad \text{for $t,s\ge0$},
\end{equation}
and have a descending chain of domains,
$  D(e^{-t'\JJgen}) \subset D(e^{-t\JJgen}) \subset H$ for $0<t<t'$.
\end{proposition}

The above domains serve as basic structures for the final value problem \eqref{JJheat-intro}.

\subsection{The Abstract Final Value Problem} \label{JJafvp-ssect}
The basic analysis is made for a
Lax--Milgram operator $A$ defined in $H$ from a $V$-elliptic sesquilinear form $a(\cdot,\cdot)$ 
in a Gelfand triple, i.e.,
in a set-up of three Hilbert spaces $V\hookrightarrow H\hookrightarrow V^*$ 
having the norms $\|\cdot\|$, $|\cdot|$ and $\|\cdot\|_*$, respectively. Hereby $V$ is the form domain of $a$. 
Specifically there are constants $C_j$ such that,  for all $u, v\in V$, one has $\| v\|_*\le C_1|v|\le C_2 \| v\|$ and $|a(u,v)|\le
C_3\|u\|\,\|v\|$ and $\Re a(v,v)\ge C_4\|v\|^2$. In fact, $D(A)$ consists of those $u\in V$ for
which $a(u,v)=\JJscal{f}{v}$ for some $f\in H$ holds for all $v\in V$; then $Au=f$. It is also recalled
that there is a bounded bijective extension $A\colon V\to V^*$ given by $\JJdual{Au}{v}=a(u,v)$ for
$u,v\in V$. (The reader may consult \cite[Ch.\ 12]{JJG09}, \cite{JJHel13} or \cite{JJChJo18ax} for more
details on the set-up and basic properties of the unbounded, but closed operator $A$ in $H$.) 

In this framework, the general final value problem is as follows: \emph{for given data  
$f\in L_2(0,T; V^*)$ and $u_T\in H$,
determine the  $u\in{\cal D}'(0,T;V)$ such that}
\begin{equation}
  \label{JJeq:fvA-intro}
  \left\{
  \begin{aligned}
  \partial_tu +Au &=f  &&\quad \text{in ${\cal D}'(0,T;V^*)$},
\\
  u(T)&=u_T &&\quad\text{in $H$}.
\end{aligned}
\right.
\end{equation}
By definition of the vector distribution space ${\cal D}'(0,T;V^*)$, cf.\ \cite{JJSwz66}, the above equation means that 
$\JJdual{u}{-\varphi'}+\JJdual{A u}{\varphi}=\JJdual{f}{\varphi}$ holds as an identity in $V^*$
for every scalar test function $\varphi\in C_0^\infty(]0,T[)$.

A wealth of parabolic Cauchy problems with homogeneous boundary 
conditions have been treated via triples $(H,V,a)$ and the ${\cal D}'(0,T;V^*)$
set-up in \eqref{JJeq:fvA-intro}; cf.\ the
work of Lions and Magenes~\cite{JJLiMa72}, Tanabe~\cite{JJTan79}, Temam~\cite{JJTem84}, Amann \cite{JJAma95}.

To compare \eqref{JJeq:fvA-intro} with the Cauchy problem for $u'+Au=f$
obtained from the initial condition $u(0)=u_0\in H$, for some $f\in L_2(0,T;V^*)$, it is recalled that there is a
unique solution $u$ in the Banach space 
\begin{equation}
  \begin{split}
  X=&L_2(0,T;V)\bigcap C([0,T];H) \bigcap H^1(0,T;V^*),
\\
  \|u\|_X=&\big(\int_0^T \|u(t)\|^2\,dt+\sup_{0\le t\le T}|u(t)|^2+\int_0^T (\|u(t)\|_*^2   +\|u'(t)\|_*^2)\,dt\big)^{1/2}.   
  \end{split}
  \label{JJeq:X}
\end{equation}
For \eqref{JJeq:fvA-intro} it would thus be natural to envisage solutions $u$ in the same space $X$.
This turns out to be true, but only under substantial further conditions on the data $(f, u_T)$. 

To formulate these, it is exploited that $\JJgen=-A$ generates an analytic semigroup $e^{-zA}$ in
$\JJB(H)$, where $z\in S_{\theta}$ for $\theta=\operatorname{arccot}(C_3/C_4)$. 
This is classical, but crucial for the entire analysis (\cite[Lem.\ 4]{JJChJo18ax} has a verification of (i), hence of (ii), in
Proposition~\ref{JJPazy'-prop}). By Proposition~\ref{JJinj-prop}, it therefore
has the inverse $(e^{-tA})^{-1}=e^{tA}$ for $t>0$.

Its domain is the Hilbert space $D(e^{tA})=R(e^{-tA})$ with $\|u\|=(|u|^2+|e^{tA}u|^2)^{1/2}$.
In \cite[Prop.\ 11]{JJChJo18ax} it was shown that a non-empty spectrum,  $\sigma(A)\ne\emptyset$,
yields strict inclusions, as one could envisage,
\begin{equation} \label{JJdom-intro}
  D(e^{t'A})\subsetneq D(e^{t A})\subsetneq H \qquad\text{ for  $0<t<t'$}.
\end{equation}
This follows from the injectivity of $e^{-tA}$, using some well-known result for semigroups that
may be found in \cite{JJPaz83}; cf.\ \cite[Thm.\ 11]{JJChJo18ax} for details.

For $t=T$ these domains enter decisively in the well-posedness result below, where
condition \eqref{JJeq:cc-intro} is a fundamental clarification for the class of final value problems \eqref{JJeq:fvA-intro}.
But it also has important implications for parabolic differential equations. 

Another ingredient is the full yield $y_f$ of the source term $f$, namely
\begin{equation} \label{JJyf-eq}
  y_f= \int_0^T e^{-(T-s)A}f(s)\,ds.
\end{equation}
Hereby it is used that $e^{-tA}$ extends to an analytic semigroup in  $V^*$,
as the extension $A\in\JJB(V,V^*)$ is an unbounded operator in the Hilbertable space $V^*$ satisfying the necessary
estimates; cf.\ \cite[Lem.\ 4]{JJChJo18ax}.
Hence $y_f$ is a priori a vector in $V^*$, but it belongs to $H$ in view of \eqref{JJeq:X},
as it is the final state of a solution of the Cauchy problem with $u_0=0$. Moreover, 
the Closed Range Theorem implies, cf.\ \cite[Prop.\ 5]{JJChJo18ax}, 
that the operator $f\mapsto y_f$ is a continuous \emph{surjection} $y_f\colon
L_2(0,T;V^*)\to H$.

These remarks on $y_f$ make it clear that the difference in \eqref{JJeq:cc-intro} is meaningful in $H$:

\begin{theorem} \label{JJintro-thm}
  Let $A$ be a $V$-elliptic Lax--Milgram operator defined from a triple $(H,V,a)$ as above.
  Then the abstract final value problem \eqref{JJeq:fvA-intro} has a solution
  $u(t)$ belonging the space $X$ in \eqref{JJeq:X}, if and only if the data
  $(f,u_T)$ belong to the subspace 
  \begin{equation}
    Y\subset L_2(0,T; V^*)\oplus H
  \end{equation}
  defined  by the condition  
  \begin{equation}
    \label{JJeq:cc-intro}
    u_T-\int_0^T e^{-(T-t)A}f(t)\,dt \ \in\  D(e^{TA}).
  \end{equation}  
In the affirmative case, the solution $u$ is uniquely determined in $X$ and 
\begin{equation}
  \label{JJeq:Y-intro}
  \begin{split}
      \|u\|_{X}& \le
  c \Big(|u_T|^2+\int_0^T\|f(t)\|_*^2\,dt+\Big|e^{TA}\big(u_T-\int_0^Te^{-(T-t)A}f(t)\,dt\big)\Big|^2\Big)^{\frac12}
\\
   &=c \|(f,u_T)\|_Y.    
  \end{split}
\end{equation}
whence the solution operator $(f,u_T)\mapsto u$ is continuous $Y\to X$. Moreover,
\begin{equation} \label{JJeq:fvp_solution}
  u(t) = e^{-tA}e^{TA}(u(T)-y_f) + \int_0^t e^{-(t-s)A}f(s) \,ds,
\end{equation}
where all terms belong to $X$ as functions of $t\in[0,T]$.
\end{theorem}

The norm on the data space $Y$ in \eqref{JJeq:Y-intro} is seen at once to be the graph norm of the composite map
\begin{equation}
  L_2(0,T; V^*)\oplus H \xrightarrow[\;]{\quad \Phi\quad} H\xrightarrow[\;]{\quad e^{TA}\quad} H
\end{equation}
given by $(f,u_T)\ \mapsto u_T-y_f\ \mapsto \ e^{TA}(u_T-y_f)$.
In fact,  \eqref{JJeq:cc-intro} means that the operator $e^{TA}\Phi$ must be defined at
$(f,u_T)$, so the data space $Y$ is its domain. Being an inverse, $e^{TA}$ is a closed operator, and so is
$e^{TA}\Phi$; hence $Y=D(e^{TA}\Phi)$ is complete. Consequently $Y$ is a Hilbertable 
space (like $V^*$).

Thus the unbounded operator $e^{TA}\Phi$ is a key ingredient in the rigorous treatment of \eqref{JJeq:fvA-intro}.
In control theoretic terms its role is to provide a unique initial state
\begin{equation}
  u(0)=e^{TA}\Phi(f,u_T)
\end{equation}
that is steered by $f$ to the final state $u(T)=u_T$ at time $T$;
cf.~\eqref{JJeq:bijection-intro} below.

Because of $e^{-(T-t)A}$ and the integral over $[0,T]$, condition
\eqref{JJeq:cc-intro} clearly involves \emph{non-local} operators in both space and time as an inconvenient
aspect\,---\,which is exacerbated by the abstract domain $D(e^{TA})$ that for longer lengths $T$ of 
the time interval gives increasingly stricter conditions; cf.\ \eqref{JJdom-intro}. 

Anyhow, \eqref{JJeq:cc-intro} is a \emph{compatibility} condition on
the data $(f,u_T)$, and thus the notion of compatibility is generalised.
For comparison it is recalled that Grubb and Solonnikov~\cite{JJGrSo90} systematically investigated
a large class of \emph{initial}-boundary problems of parabolic (pseudo-)differential equations
and worked out compatibility conditions, which are necessary and sufficient for
well-posedness in full scales of anisotropic $L_2$-Sobolev spaces. Their conditions are explicit
and local at the curved corner $\partial\Omega\times\{0\}$, except for 
half-integer values of the smoothness $s$ that were shown to require so-called coincidence, which 
is expressed in integrals over the product of the two boundaries $\{0\}\times\Omega$ and
$\,]0,T[\,\times\,\partial\Omega$; hence it also is a non-local condition.  
However, whilst their conditions are decisive for the solution's regularity, 
the above condition \eqref{JJeq:cc-intro} is crucial for the \emph{existence} question; cf.\ the theorem.

\bigskip
Previously, uniqueness in \eqref{JJeq:fvA-intro} was shown by Amann~\cite[Sect.~V.2.5.2]{JJAma95} in a $t$-dependent set-up, 
but injectivity of the map $u(0)\mapsto u(T)$ was proved much earlier for problems with $t$-dependent
sesquilinear forms by Lions and Malgrange~\cite{JJLiMl60}.

Showalter~\cite{JJSho74} strived to characterise the possible $u_T$ via Yosida
approximations for $f=0$ and $A$ having half-angle $\frac\pi4$.
Invertibility of $e^{-tA}$  was claimed for this purpose in \cite{JJSho74} for such $A$ (but, as
mentioned, not quite obtained).

To make a few more remarks, it is noted that the proof given below exploits that the solution $u$
also in this set-up necessarily is given by Duhamel's principle, or the 
variation of constants formula, for the analytic semigroup $e^{-tA}$ in $V^*$, 
\begin{equation}
  \label{JJeq:bijection-intro}
  u(t)=e^{-tA}u(0)+ \int_0^t e^{-(t-s)A}f(s)\,ds.
\end{equation}
For $t=T$ this yields a \emph{bijection} $u(0)\leftrightarrow u(T)$ between the initial
and terminal states; in particular backwards uniqueness of the solutions holds in the large class $X$.
Of course, this relies crucially on the invertibility of $e^{-tA}$ in Proposition~\ref{JJinj-prop}.

Now, \eqref{JJeq:bijection-intro} also shows that $u(T)$ consists of two 
parts, that differ radically even when $A$ has nice properties: 
First, the integral amounts to $y_f$ for $t=T$, and by the mentioned surjectivity this terms can be
\emph{anywhere} in $H$.

Secondly, $e^{-tA}u(0)$ solves $u'+Au=0$, and for $u(0)\ne0$ there is 
the precise property in non-selfadjoint dynamics that the ``height'' function $h(t)= |e^{-tA}u(0)|$ is
\begin{equation}
  \text{strictly positive ($h>0$), strictly decreasing ($h'<0$)  and \emph{strictly convex}}  .
\end{equation}
Whilst this holds if $A$ is self-adjoint or normal, it was emphasized in \cite{JJChJo18ax}
that it suffices that $A$ is just hyponormal. Recently this was
followed up by the author in \cite{JJ18logconv}, where the stronger 
logarithmic convexity of $h(t)$ was proved \emph{equivalent} to the weaker property of $A$ that 
$2(\Re\JJscal{A x}{x})^2\le \Re\JJscal{A^2x}{x}|x|^2+|A x|^2|x|^2$ for $x\in D(A^2)$.

The stiffness inherent in \emph{strict} convexity reflects that $u(T)=e^{-TA}u(0)$ is
confined to a dense, but very small space, as by the analyticity
\begin{equation}
  \label{JJDAn-cnd}
  u(T)\in \textstyle{\bigcap_{n\in\JJN}} D(A^n).
\end{equation}
For $u'+Au=f$, the possible final data
$u_T$ will hence be a sum of an arbitrary vector $y_f$ in $H$ and a
term $e^{-TA}u(0)$ of great stiffness, cf.~\eqref{JJDAn-cnd}. Thus $u_T$ can be prescribed in the affine space
$y_f+D(e^{TA})$. As any $y_f\ne0$ will shift $D(e^{TA})\subset H$
in some arbitrary direction, $u(T)$ can be expected \emph{anywhere} in $H$ (unless $y_f\in D(e^{TA})$ is known).
So neither $u(T)\in D(e^{TA})$ nor \eqref{JJDAn-cnd} can be expected to hold for $y_f\ne0$---not even
if $|y_f|$ is much smaller than $|e^{-TA}u(0)|$. In view of this conclusion, it seems best for final value problems to
consider inhomogeneous problems from the very beginning.

\section{Proof of Theorem~\ref{JJintro-thm}}
  \label{JJaproof-sect}
The point of departure is the following well-known result, which is formulated as a theorem here
only to indicate that it is a cornerstone in the proof. It is also emphasized that the Lax--Milgram
operator $A$ need not be selfadjoint.
 
\begin{theorem}  \label{JJthm:Temam}
Let $V$ be a separable Hilbert space with $V \subseteq H$ algebraically, topologically and densely,
and let $A$ denote the Lax--Milgram operator induced by a $V$-elliptic
sesquilinear form, as well as its extension $A\in\JJB(V,V^*)$, cf.\ Section~\ref{JJafvp-ssect}. 
When $u_0 \in H$ and $f \in L_2(0,T; V^*)$ are given, then the Cauchy problem 
\begin{equation}
  \left\{
  \begin{aligned}
  \partial_tu +Au &=f  \quad \text{in ${\cal D}'(0,T;V^*)$},
\\
  \qquad u(0)&=u_0 \quad\text{in $H$}, 
  \end{aligned}
  \right.
  \label{JJivp-id}
\end{equation}
has a uniquely determined solution $u(t)$ belonging to the space $X$ in \eqref{JJeq:X}.
\end{theorem}

This is a special case of a classical result of Lions and Magenes~\cite[Sect.~3.4.4]{JJLiMa72} on 
$t$-dependent forms $a(t;u,v)$. The conjunction of  the properties 
$u\in L_2(0,T;V)$ and $u'\in L_2(0,T;V^*)$, which appears in \cite{JJLiMa72}, is clearly equivalent to the property
in \eqref{JJeq:X} that $u$ belongs to the intersection of $L_2(0,T,V)$ and $H^1(0,T;V^*)$.

To clarify a redundancy, it is first noted that in Theorem~\ref{JJthm:Temam}
the solution space $X$ is a Banach space, which can have its norm in 
\eqref{JJeq:X} written in the form
\begin{align}  \label{JJeq:Xnorm}
\|u\|_{X} = \big(\|u\|^2_{L_2(0,T;V)} + \sup_{0 \leq t \leq T}|u(t)|^2 + \|u\|^2_{H^1(0,T;V^*)}\big)^{1/2}.
\end{align}
Here there is a well-known inclusion $L_2(0,T;V)\cap H^1(0,T;V^*)\subset C([0,T];H)$ and an associated
Sobolev inequality for vector functions
(\cite{JJChJo18ax} has an elementary proof)
  \begin{equation} \label{JJSobolev-ineq}
  \sup_{0\le t\le T}| u(t)|^2\le (1+\frac{C_2^2}{C_1^2T})\int_0^T \|u\|^2\,dt+\int_0^T \|u'\|_*^2\,dt.
  \end{equation}
Hence one can safely omit the space $C([0,T];H)$ in  \eqref{JJeq:X}.
Likewise $\sup|u|$ can be removed from $\| \cdot\|_{X}$,
as one obtains an equivalent norm (similarly $\int_0^T\|u(t)\|_*^2\,dt$ is redundant
in \eqref{JJeq:X}). 
Thus $X$ is more precisely a Hilbertable space;
but \eqref{JJeq:X} is kept as stated in order to emphasize the properties of the solutions.

However, two refinements of the above theory is needed.
For one thing, the next result yields well-posedness of \eqref{JJivp-id}, which is a well-known corollary to the
proofs in \cite{JJLiMa72}. 
But a short explicit argument is also possible:

\begin{proposition}  \label{JJpest-prop}
The unique solution $u$ of \eqref{JJivp-id}, given by
Theorem \ref{JJthm:Temam}, depends
as an element of $X$ continuously on the data $(f,u_0) \in L_2(0,T;V^*)\oplus H$, i.e.
\begin{align}   
\|u\|^2_{X} \leq c (|u_0|^2 + \|f\|^2_{L_2(0,T;V^*)}).
\end{align}
\end{proposition}

Here and in the sequel, $c$ denotes as usual a constant in $\,]0,\infty[\,$ of unimportant value,
which moreover may change from place to place.

\begin{proof}
As $u \in L_2(0,T;V)$ while $u',f$ and $A u$ belong to the dual space $L_2(0,T;V^*)$, one has the identity 
$\Re\JJdual{\partial_t u}{u} + \Re\JJdual{A u}{u} = \Re\JJdual{f}{u}$ in $L_1$.
Now, a classical regularisation yields $\partial_t |u|^2=2\Re\JJdual{\partial_t u}{u}$, 
so by Young's inequality and the $V$-ellipticity,
\begin{align}
  \partial_t |u|^2 +  2C_4 \|u\|^2 \leq 2 |\JJdual{f}{u}| \leq C_4^{-1} \|f\|_{*}^2 + C_4 \|u\|^2.
\end{align}
Using again that $|u(t)|^2$ and $\partial_t |u(t)|^2 $ are in $L_1(0,T)$, integration yields
\begin{align}
  |u(t)|^2 + C_4 \int_0^t \|u(s)\|^2 \,ds \leq |u_0|^2 + C_4^ {-1}\|f\|_{L_2(0,T;V^*)}^2.
\end{align}
This gives $\|u\|_{L_2(0,T;V)}^2 \leq C_4^{-1}|u_0|^2 + C_4^{-2}\|f\|_{L_2(0,T;V^*)}^2$ for the
first term in $\|u\|_X$. 
As $u$ solves \eqref{JJivp-id}, it is clear that
$\|\partial_t u(t) \|_{*}^2 \leq (\|f(t)\|_{*} + \|A u \|_{*})^2 $,
hence
\begin{align}
  \int_0^T \|\partial_t u(t) \|_{*}^2 \,dt \leq 2 \int_0^T \|f(t)\|_{*}^2 \,dt 
  + 2 \|A\|_{\JJB(V,V^*)}^2  \int_0^T \|u\|^2 \,dt ,
\end{align}
where $\int_0^T \|u\|^2 \,dt$ is estimated above. Finally $\sup|u|$ can be covered via \eqref{JJSobolev-ineq}.
\end{proof}

Secondly, as $e^{-tA}$ extends to an analytic semigroup in $V^*$, cf.\ \cite[Lem.\ 5]{JJChJo18ax}, so
that $e^{-tA}f(t)$ is defined, Theorem~\ref{JJthm:Temam} can be supplemented by the 
explicit expression:
\begin{equation} \label{JJu-id}
  u(t) = e^{-tA}u_0 + \int_0^t e^{-(t-s)A}f(s) \,ds \qquad\text{for } 0\leq t\leq T.
\end{equation}
This is of course the Duhamel formula, but even for analytic semigroups
the proof that it does give a solution requires $f(t)$ to be H\"older continuous (\cite[Cor.\ 4.3.3]{JJPaz83} is slightly
more general, though), whereas above $f\in L_2(0,T;V^*)$. As the present space $X$
even contains non-classical solutions, \eqref{JJu-id} requires a new proof here---but it suffices to reinforce the
usual argument by the injectivity of $e^{-tA}$ in Proposition~\ref{JJinj-prop}:
 
\begin{theorem}   \label{JJDuhamel-thm}
The unique solution $u$ in $X$ provided by Theorem~\ref{JJthm:Temam} is given by \eqref{JJu-id},
where each of the three terms belongs to $X$.
\end{theorem}
\begin{proof}
Once \eqref{JJu-id} has been shown, Theorem~\ref{JJthm:Temam} yields for $f=0$ that
$u(t)\in X$, hence $e^{-tA}u_0\in X$. For general $(f,u_0)$ the last term
containing $f$ is then also in $X$.

To obtain \eqref{JJu-id} in the above context, note that all terms in 
$\partial_t u+A u=f$ are in $L_2(0,T;V^*)$. Therefore  $e^{-(T-t)A}$
applies to both sides as an integration factor, so 
\begin{equation}
  \partial_t(e^{-(T-t)A}u(t))=e^{-(T-t)A}\partial_t u(t)+ e^{-(T-t)A}A u(t)=e^{-(T-t)A}f(t).
\end{equation}
Indeed, $e^{-(T-t)A}u(t)$ is in $L_1(0,T;V^*)$ and 
its derivative in ${\cal D}'(0,T;V^*)$ follows a Leibniz rule, as one can prove by regularisation
since $u(t)\in V=D(A)$ for $t$ a.e.

The right-hand side above is in $L_2(0,T;V^*)$, hence in $L_1(0,T;V^*)$, so when the Fundamental Theorem for vector functions 
(cf.\ \cite[Lem.\ III.1.1]{JJTem84}) is applied and followed by commutation of $e^{-(T-t)A}$ with the
integral (via Bochner's identity), 
\begin{equation} \label{JJeq:identityT}
  \begin{split}
    e^{-(T-t)A}u(t)&=e^{-TA}u_0+\int_0^t e^{-(T-s)A}f(s)\,ds
\\
  &=e^{-(T-t)A}e^{-tA}u_0+e^{-(T-t)A}\int_0^t e^{-(t-s)A}f(s)\,ds. 
  \end{split}
\end{equation}
Since $e^{-(T-t)A}$ is injective, cf.~Proposition~\ref{JJinj-prop}, \eqref{JJu-id} now results at once. 
\end{proof}

As all terms in \eqref{JJu-id} are in $C([0,T];H)$, we may safely evaluate at $t=T$, which in view of
\eqref{JJyf-eq} gives that
$u(T)=e^{-TA}u(0)+y_f$; this is the flow map $u(0)\mapsto u(T)$. 
Owing to the invertibilty of $e^{-TA}$ once again, this flow is inverted  by
\begin{equation}  \label{JJu0uT-id}
  u(0)=e^{TA}(u(T)-y_f).  
\end{equation}
In other words, the solutions in $X$ to $u'+Au=f$ are for fixed $f$ parametrised by the initial
states $u(0)$ in $H$. Departing from this observation, one may give an intuitive

\begin{proof}[Proof of Theorem 1]
When \eqref{JJeq:fvA-intro} has a solution $u \in X$, then $u(T)=u_T$ is reached from
the initial state $u(0)$ determined from the bijection in \eqref{JJu0uT-id}. This 
gives that $u_T-y_f = e^{-TA} u(0)\in D(e^{TA})$, so  \eqref{JJeq:cc-intro}
is necessary. 

In case $u_T$, $f$ do fulfill \eqref{JJeq:cc-intro},
then $u_0 = e^{TA}(u_T - y_f)$ is a well-defined vector in $H$, so 
Theorem~\ref{JJthm:Temam} yields a
function $u\in X$ solving $u' +Au = f$ and $u(0)=u_0$. 
By \eqref{JJu0uT-id} 
this $u$ has final state $u(T)=e^{-TA}e^{TA}(u_T-y_f)+y_f=u_T$, hence solves \eqref{JJeq:fvA-intro}.

In the affirmative case, one obtains \eqref{JJeq:fvp_solution} by insertion
of formula \eqref{JJu0uT-id} for $u(0)$ into \eqref{JJu-id}. That each term in \eqref{JJeq:fvp_solution}
is a function belonging to $X$ was seen in Theorem~\ref{JJDuhamel-thm}.

Uniqueness of $u$ in $X$ is obvious from the right-hand side of \eqref{JJeq:fvp_solution}. 
The solution can hence be estimated in $X$ by insertion of \eqref{JJu0uT-id} into the inequality
in Proposition~\ref{JJpest-prop}, which gives $\|u\|_X^2\le
c(|e^{TA}(u_T-y_f)|^2+\|f\|_{L_2(0,T;V^*)}^2)$. Here one may add $|u_T|^2$ on the right-hand side
to arrive at the expression for $\|(f,u_T)\|_Y$ given in Theorem 1.
\end{proof}

\begin{remark}
 The above arguments seem to extend to Lax--Milgram operators $A$ that are only
 $V$-coercive, i.e.\ fulfil $\Re a(u,u)\ge C_4\|u\|^2-k|u|^2$ for $u\in V$. In fact, it was
 observed already in \cite{JJLiMa72} that Theorem~\ref{JJthm:Temam} holds \emph{verbatim} for such $A$,
 for since $A+k$ is $V$-elliptic, the unique solvability in $X$ of $v'+(A+k)v=e^{-kt}f$, $v(0)=u_0$ yields a
 unique solution $u=e^{kt}v$ of $u'+Au=f$, $u(0)=u_0$. Since the same translation trick gave the
 improved version of Proposition~\ref{JJinj-prop}, also $V$-coercive $A$ generate
 analytic semigroups of injections; so the proofs of the Duhamel formula and of
 Theorem~\ref{JJheat-intro} seem applicable in their present form. However, the estimates in Proposition~\ref{JJpest-prop} need to
 be modified using Gr{\"o}nwall's lemma. The details of this are left for future work.
 (Added in proof: An elaboration of the results indicated in this remark will appear in the forthcoming paper \cite{JJ19cor}.)
\end{remark}

\begin{remark}
Recently Almog, Grebenkov, Helffer, Henry \cite{JJAlHe15,JJGrHelHen17,JJGrHe17} studied
variants of the complex Airy operator via triples $(H,V,a)$, and 
Theorem~\ref{JJheat-intro} is expected to extend to final value problems for those of their realisations that have
non-empty spectrum.
\end{remark}

\section{The Heat Problem with Final Time Data}
  \label{JJheat-sect}

To follow up on Theorem~\ref{JJintro-thm}, it is now applied to the heat equation and its final value problem.
In the sequel $\Omega$ stands for a $C^\infty$ smooth, open bounded set in $\JJRn$,
$n\ge2$ as described in \cite[App.~C]{JJG09}. In particular $\Omega$  is
locally on one side of its boundary $\Gamma=\partial \Omega$.
For such sets  we consider the problem of finding the $u$ satisfying
\begin{equation}  \label{JJeq:heat_fvp}
\left\{
\begin{aligned}
  \partial_tu(t,x) -\Delta u(t,x) &= f(t,x) &&\text{ in } Q= ]0,T[ \times \Omega ,
\\
  \gamma_0 u(t,x) &= g(t,x) && \text{ on } S= ]0,T[ \times \Gamma,
\\
   r_T u(x) &= u_T(x) && \text{ at } \left\{ T \right\} \times \Omega.
\end{aligned}
\right .
\end{equation}
Hereby the trace of functions on $\Gamma$ is written in the operator notation $\gamma_0u=
u|_{\Gamma}$. Similarly $\gamma_0$ is also used for traces on $S$, while $r_T$ denotes the trace 
at $t=T$.

Moreover, $H^1_0(\Omega)$ is the subspace obtained by closing 
$C_0^\infty(\Omega)$ in the Sobolev space $H^1(\Omega)$. Dual to this one has $H^{-1}(\Omega)$,
which identifies with the set of restrictions to $\Omega$ from $H^{-1}(\JJRn)$, endowed with the
infimum norm; Chapter 6 and Remark 9.4 in \cite{JJG09} could be references for this and basic
facts on boundary value problems.

\subsection{The Boundary Homogeneous Case} 
In case  $g \equiv 0$ in \eqref{JJeq:heat_fvp}, the main result in Theorem~\ref{JJintro-thm} applies for
\begin{align}
  V = H_0^1(\Omega),  \quad H = L_2(\Omega), \quad  V^* = H^{-1}(\Omega).
\end{align}
Indeed, the boundary condition $\gamma_0u=0$ is then imposed via the condition that $u(t)\in V$ for all $t$,
or rather through use of the Dirichlet realization of the Laplacian $-\Delta_{\gamma_0}$ 
(denoted by $\JJmlap_D$ in the introduction), which is
the Lax--Milgram operator $A$ induced by the triple
$(L_2(\Omega),H_0^1(\Omega),s)$ for 
\begin{align}
  s(u,v) = \sum_{j=1}^n \JJscal{\partial_j u}{\partial_j v}_{L_2(\Omega)}. 
\end{align}
In fact, Poincar\' e's inequality yields that $s(u,v)$ is
$H_0^1(\Omega)$-elliptic and symmetric, so $A = -\JJlap_{\gamma_0}$ is a selfadjoint
unbounded operator in $L_2(\Omega)$, with $D(-\JJlap_{\gamma_0})\subset H^1_0(\Omega)$.  

Hence $-A = \JJlap_{\gamma_0}$ generates an analytic semigroup of operators 
$e^{t\JJlap_{\gamma_0}}$ in $\JJB(L_2(\Omega))$, and the bounded bijective extension 
$\JJlap\colon H^{1}_0(\Omega) \rightarrow H^{-1}(\Omega)$ 
induces the analytic semigroup $e^{t\JJlap}$ on $ V^*=H^{-1}(\Omega)$.
As done previously, we set $(e^{t\JJlap_{\gamma_0}})^{-1} = e^{-t\JJlap_{\gamma_0}}$.

Moreover, when $g=0$ in \eqref{JJeq:heat_fvp}, then the solution and data spaces amount to
\begin{align}
  X_0 &= L_2(0,T;H^1_0(\Omega)) \bigcap C([0,T]; L_2(\Omega)) \bigcap H^1(0,T; H^{-1}(\Omega)),
\label{JJX0-id}
\\  
  Y_0&= \left\{ (f,u_T) \in L_2(0,T;H^{-1}(\Omega)) \oplus L_2(\Omega) \Bigm|  
                  u_T - y_f \in D(e^{-T\Delta_{\gamma_0}}) \right\}.
\label{Y0-id}
\end{align}
Here, with $y_f$ given in \eqref{JJyf-eq}, the data norm from Theorem~\ref{JJintro-thm} specialises to
\begin{align}
 \| (f,u_T) \|_{Y_0}^2  
  = \int_0^T\|f(t)\|^2_{H^{-1}(\Omega)}\,dt 
  + \int_\Omega(|u_T|^2+|e^{-T\Delta_{\gamma_0}}(u_T - y_f )|^2)\,dx.
\end{align} 

Now Theorem~\ref{JJintro-thm} straightforwardly gives the following result, which
first appeared in \cite{JJChJo18ax} even though the problem is entirely classical:

\begin{theorem}  \label{JJheat0-thm}
Let $A=\JJmlap_{\gamma_0}$ be the Dirichlet realization of the Laplacian in $\Omega$ and 
$\JJmlap$ its extension, as introduced above.
When $g=0$ in the final value problem \eqref{JJeq:heat_fvp} and 
$f \in L_2(0,T;H^{-1}(\Omega))$, $u_T \in L_2(\Omega)$, 
then there exists a solution $u$ in $X_0$ of \eqref{JJeq:heat_fvp} if and only if the data $(f,u_T)$
are given in $Y_0$, i.e.\ if and only if
\begin{equation}  \label{heat-ccc}
  u_T - \int_0^T e^{(T-s)\JJlap}f(s) \,ds\quad \text{ belongs to }\quad D(e^{-T \JJlap_{\gamma_0}}). 
\end{equation}
In the affirmative case, $u$ is uniquely determined in $X_0$ and 
$\|u\|_{X_0} \leq c \| (f,u_T) \|_{Y_0}$.
Furthermore the difference  in \eqref{heat-ccc} equals 
$e^{T\Delta_{\gamma_0}}u(0)$ in $L_2(\Omega)$. 
\end{theorem}

\subsection{The Inhomogeneous Case}

When $g\ne0$ on the surface $S$, cf.\ \eqref{JJeq:heat_fvp},
then the solution $u(t,\cdot)$ belongs to the full Sobolev space $H^1(\Omega)$ for each $t>0$, so  here the
solution space is denoted by $X_1$, 
\begin{align} \label{JJX1-id}
 X_1 = L_2(0,T; H^1(\Omega)) \bigcap C([0,T];L_2(\Omega)) \bigcap H^1(0,T;H^{-1}(\Omega)).
\end{align}
Clearly $X_1$ is a Banach space when normed analogously to \eqref{JJeq:Xnorm},
\begin{align} \label{JJXheat-nrm}
  \| u \|_{X_1}= (\|u\|_{L_2(0,T; H^1(\Omega))}^2 + \sup_{0\leq t \leq T} \|u(t)\|_{L_2(\Omega)}^2 
                   + \|u\|_{H^1(0,T; H^{-1}(\Omega))}^2)^{1/2}.
\end{align}
Here $H^1$, $H^{-1}$ are not dual on $\Omega$, so the previously mentioned redundancy 
does not extend to the term $\sup_{[0,T]} \|u\|_{L_2}$ above.

For $g\ne0$ the standard strategy is, of course,  to (try to)
reduce to a semi-homogeneous problem by choosing $w$ so that
$\gamma_0 w = g$ on $S$, using the classical

\begin{lemma}  \label{JJwKg-lem}
$\gamma_0\colon H^1(Q)\to H^{1/2}(S)$ is a continuous surjection having a
bounded right inverse $\tilde{K}_0$, that is,  $\gamma_0\tilde{K}_0 g=g$
for every $g\in H^{1/2}(S)$ and
\begin{align}  \label{JJeq:wKg}
  \|w\|_{H^1(Q)}=\|\tilde{K}_0 g\|_{H^1(Q)} \leq c \|g\|_{H^{1/2}(S)}.
\end{align}
\end{lemma}
In lack of a reference with details, the reader is referred to \cite{JJChJo18ax} for a sketch of how
this lemma follows using standard techniques.

Another preparation is based on the fine theory of the elliptic problem
\begin{equation} \label{JJellDir-id}
  \JJmlap u= f,\quad \gamma_0 u=g.
\end{equation}
Indeed, it exploited below that
${\cal Q}=I-\JJlap_{\gamma_0}^{-1}\JJlap$ is a well-known projection in $H^1(\Omega)$, along
$H^1_0(\Omega)$ and onto the closed subspace of harmonic $H^1$-functions, 
\begin{equation}  \label{JJZ1-id}
  Z(-\Delta)=\{\, u\in H^1(\Omega)\mid -\Delta u=0\,\}. 
\end{equation}
To recall this, \eqref{JJellDir-id} may conveniently be treated via the matrix operator  
$\left(\begin{smallmatrix} \JJmlap\\ \gamma_0\end{smallmatrix}\right)$,
with an inverse in row form 
$\left(\begin{smallmatrix} \JJmlap_{\gamma_0}^{-1}& K_0\end{smallmatrix}\right)$ that applies to the data
$\left(\begin{smallmatrix} f\\ g\end{smallmatrix}\right)$,
for then the basic composition formulae appear in two identities on $H^1(\Omega)$ and
$H^{-1}(\Omega)\oplus H^{1/2}(\Gamma)$, 
\begin{align}
  I &= \begin{pmatrix}  \JJmlap_{\gamma_0}^{-1}& K_0\end{pmatrix}
   \begin{pmatrix} -\JJlap \\ \gamma_0\end{pmatrix}
    =\JJlap_{\gamma_0}^{-1}\JJlap+K_0\gamma_0,
\label{JJDir1-id}
\\
  \begin{pmatrix} I&0\\ 0& I\end{pmatrix}
  &=\begin{pmatrix} \JJmlap \\ \gamma_0\end{pmatrix}\begin{pmatrix}  \JJmlap_{\gamma_0}^{-1}& K_0\end{pmatrix}
   =\begin{pmatrix}  \JJlap\JJlap_{\gamma_0}^{-1}& \JJlap K_0\\ -\gamma_0\JJlap_{\gamma_0}^{-1}& \gamma_0K_0\end{pmatrix}.
\label{JJDir2-id}
\end{align}
In particular the first formula yields that ${\cal Q}=I\JJmlap_{\gamma_0}^{-1}\JJlap=K_0\gamma_0$ on $H^1(\Omega)$. 
However, it should be emphasized that the simplicity of the formulas \eqref{JJDir1-id} and~\eqref{JJDir2-id}
relies on a specific choice of $K_0$, which is recalled here:
 
As $\JJlap_{\gamma_0}=\Delta\big|_{H^1_0}$ holds in the distribution sense, ${\cal P}= \JJlap_{\gamma_0}^{-1} \JJlap $ clearly fulfils
${\cal P}^2=\cal P$, it is bounded $H^1\to H_0^1$ and equals $I$ on $H^1_0$, so $\cal P$ is the projection onto
$H_0^1(\Omega)$ along its null space, which clearly is the space in \eqref{JJZ1-id}.
Therefore $H^1$ is a direct sum,
\begin{align} 
  H^1(\Omega) = H_0^1(\Omega) \dotplus Z(\JJmlap),
\end{align}
so that ${\cal Q} = I-{\cal P}=I\JJmlap_{\gamma_0}^{-1} \JJlap$ is the projection onto $Z(\JJmlap)$ along $H_0^1(\Omega)$.

Since $\gamma_0\colon H^1(\Omega) \rightarrow H^{1/2}(\Gamma)$ is surjective with null-space $H^1_0$,
it has an inverse $K_0$ on the complement $Z(\JJmlap)$, which by the open mapping principle is bounded  
$K_0 \colon H^{1/2}(\Gamma) \rightarrow Z(\JJmlap) \hookrightarrow H^1(\Omega)$. Since in  this
construction $K_0=(\gamma_0|_{Z(\JJmlap)})^{-1}$, it is  also a right-inverse, i.e.\ $\gamma_0 K_0= I_{H^{1/2}(\Gamma)}$. The rest of
\eqref{JJDir2-id} now follows at once. 

Moreover, since $\gamma_0{\cal P}=\gamma_0\JJlap_{\gamma_0}^{-1} \JJlap=0$,
the definition of $\cal Q$  gives \eqref{JJDir1-id} thus:
\begin{align}   \label{JJeq:KgQ}
  K_0 \gamma_0 = K_0 \gamma_0 ({\cal P}+{\cal Q}) = K_0 \gamma_0 {\cal Q} = I_{Z(\JJmlap)} {\cal Q}
  = {\cal Q}=I\JJmlap_{\gamma_0}^{-1} \JJlap.
\end{align}
\begin{remark}
$K_0$ is an example of a Poisson operator;
these are amply discussed within the pseudo-differential boundary operator calculus, for example in \cite{JJG1}.
\end{remark}

\begin{remark} 
  \label{JJHsS-rem}
The $H^s(S)$-norm can be chosen so that this is a Hilbert space; cf.\ the use of local
coordinates in \cite[(8.10)]{JJG09}. 
The vast subject of Sobolev spaces on $C^\infty$ surfaces generally requires distribution
densities as explained in \cite[Sect.\ 6.3]{JJH} (cf.\ \cite[Sect.\ 8.2]{JJG09} for a
concise review). But the surface measure on $S$ induces a well-known identification of densities with
distributions on $S$, and  within this framework, a systematic exposition can be found
in \cite[Sect.\ 4]{JJJoMHSi3}, albeit in an $L_p$-setting with
anisotropic mixed-norm Triebel--Lizorkin spaces $F^{s,\vec a}_{\vec p,q}(S)$, which 
are the correct boundary data spaces for parabolic problems having different
integrability properties  in $x$ and $t$. Cf.\ also the application to the heat problem \eqref{JJeq:heat_fvp}  
in \cite[Sect.\ 6.5]{JJJoMHSi3}, and the more detailed discussion in \cite[Ch.\ 7]{JJSMHphd}.
\end{remark}

Now, when splitting the solution of \eqref{JJeq:heat_fvp} as $u=v+w$ for $w=\tilde K_0g$, cf.\ Lemma~\ref{JJwKg-lem},
then $v$ should satisfy $v'\JJmlap v=\tilde f$, $\gamma_0 v=0$ and $r_T v=\tilde u_T$ for data
\begin{equation}
  \tilde f=f-(\partial_t w-\Delta w),\qquad \tilde u_T=u_T-r_Tw.
\end{equation}
For this problem to be solved by $v$, \eqref{heat-ccc} stipulates that
$D(e^{-T\Delta_{\gamma_0}})$ should contain
\begin{equation}
  \tilde u_T-y_{\widetilde f}=u_T-y_f -(r_Tw-y_{\partial_t w \JJmlap w}).
\end{equation}
But the presence of the term $r_Tw-y_{\partial_t w\JJmlap w}$ makes it impossible just to
transfer condition \eqref{heat-ccc} of being a member of  
$D(e^{-T\Delta_{\gamma_0}})$ from $\tilde u_T-y_{\widetilde f}$ to $u_T-y_f$. Thus the
compatibility condition \eqref{heat-ccc} destroys the trick of reducing to homogeneous boundary
conditions, despite the linearity of the problem. 

To find the correct compatibility conditions on $(f,g,u_T)$, the strategy in \cite{JJChJo18ax} was
(with hindsight) 
to use Lemma~\ref{JJwKg-lem} to get a solution  formula for the corresponding linear
\emph{initial} value problem instead. 
This is motivated by the fact that, for the present space
$X_1$  of low regularity, no compatibility condition is needed for this: 
\begin{equation}  \label{JJeq:heat_ivp}
  \partial_tu -\Delta u = f \quad\text{ in } Q,
\qquad
  \gamma_0 u = g  \quad \text{ on } S,
\qquad
  r_0 u = u_0 \quad \text{ at } \{ 0 \} \times \Omega.
\end{equation}
(For general background material on \eqref{JJeq:heat_ivp} the reader could consult Section~III.6 in
\cite{JJABHN11}; and \cite{JJGrSo90} for the fine theory including compatibility conditions.)

Similarly to Theorem~\ref{JJthm:Temam} and Proposition~\ref{JJpest-prop}, one may depart from well-posedness of \eqref{JJeq:heat_ivp}.
While this is well known \emph{per se}, an explanation is given to
account below for the crucial existence of an improper integral showing up when $g\ne0$ in \eqref{JJeq:heat_fvp}.

\begin{proposition}  \label{JJheat-prop}
The heat initial value problem \eqref{JJeq:heat_ivp} has a unique solution $u\in X_1$ for given 
$f\in L_2(0,T;H^{-1}(\Omega))$, $g\in H^{1/2}(S)$, $u_0\in L_2(\Omega)$,
and there is an estimate
\begin{align}  \label{JJheat-ineq}
  \|u\|_{X_1}^2 \leq 
  c ( \|u_0\|_{L_2(\Omega)}^2 + \|f\|_{L_2(0,T;H^{-1}(\Omega))}^2 + \|g\|_{H^{1/2}(S)}^2 ).
\end{align}
\end{proposition}

\begin{proof}
Let $I=\,]0,T[\,$. Setting $w=\tilde{K}_0 g$ by means of Lemma~\ref{JJwKg-lem}, we tentatively
write $u=v+w$ for some $v\in X_0$ solving \eqref{JJeq:heat_ivp} for data 
\begin{align}
  \tilde{f} = f - (\partial_t - \Delta)w,\qquad \tilde g=0,\qquad  \tilde{u}_0= u_0 - w(0).
\end{align} 
Here $w(0)$ is well defined, as $w \in H^1(Q)\subset H^1([0,T] ; L_2(\Omega)) \subset C([0,T] ;
L_2(\Omega))$ by a Sobolev embedding, which (e.g.\ as in \cite[Rem.\ 4]{JJChJo18ax}) also
gives the estimate
\begin{align} 
  \sup_{0\leq t \leq T} \|w(t)\|_{L_2(\Omega)} 
  \leq c (\|w\|_{L_2(I;L_2(\Omega))} + \|\partial_t w\|_{L_2(I;L_2(\Omega))}) \le c \|w\|_{H^1(Q)}^2.
\end{align}
To show that $w\in X_1$ and $\|w\|_{X_1}^2 \leq c \|w\|_{H^1(Q)}^2$, it is noted that
estimates of the two remaining terms in $\|w\|_{X_1}^2$, cf.\ \eqref{JJXheat-nrm}, can be read off from the obvious inequalities 
\begin{equation}  \label{JJw-ineq} 
  \| w'\|_{L_2(I;H^{-1})}^2 + \|\Delta w\|_{L_2(I;H^{-1})}^2 
  \leq c'\|w\|_{H^1(I;L_2)}^2 + c'' \|w\|_{L_2(I;H^1)}^2 
  \leq c \|w\|_{H^1(Q)}^2.
\end{equation}
This also entails $\tilde{f} \in L_2(0,T;H^{-1}(\Omega))$, and $\tilde{u}_0 \in L_2(\Omega)$,
so by Theorem~\ref{JJthm:Temam}, the boundary homogeneous problem 
for $v$ has a solution in $X_0$; cf.\ \eqref{JJX0-id}. 
Hence \eqref{JJeq:heat_ivp} has the solution $u=v+w$ in $X_1$;  its uniqueness is easily carried over from
Theorem~\ref{JJthm:Temam}. 

Finally the estimate in Theorem~\ref{JJheat0-thm} (that is a consequence of Proposition~\ref{JJpest-prop}) gives
\begin{align}
  \|u\|_{X_1}^2 & \leq 2 ( \|v\|_{X_0}^2 + \|w\|_{X_1}^2)  
  \leq c (\|\tilde{u}_0\|_{L_2(\Omega)}^2 + \|\tilde{f}\|_{L_2(I;H^{-1}(\Omega))}^2 +
  \|w\|_{X_1}^2) 
\nonumber\\
  & \leq c (\|u_0\|_{L_2(\Omega)}^2 + \|f\|_{L_2(I;H^{-1}(\Omega))}^2 
  + \|w' \JJmlap w \|_{L_2(I;H^{-1})}^2 + \|w\|_{H^1(Q)}^2),
\end{align}
which via  \eqref{JJw-ineq} and \eqref{JJeq:wKg} entails the stated estimate \eqref{JJheat-ineq}.
\end{proof}

Obviously the formula in Theorem~\ref{JJDuhamel-thm} now applies directly to the function $v=u-w$
in the above proof, which as a crucial addendum to Proposition~\ref{JJheat-prop} yields 
\begin{equation} \label{JJeq:uw}
  u(t)-w(t)= e^{t\Delta_{\gamma_0}}(u_0-w(0))+\int_0^t e^{(t-s)\JJlap}(f-(\partial_s-\Delta)w)\,ds.
\end{equation}
The  next step is to rewrite the contributions from $w$ so that $g=\gamma_0 w$
can be reintroduced.
First a regularisation of $w$ in $H^1(0,t;L_2(\Omega))$ leads to the Leibniz rule
\begin{equation}  \label{JJw-id}
   \partial_s(e^{(t-s)\Delta_{\gamma_0}}w(s))
   =e^{(t-s)\Delta_{\gamma_0}}\partial_sw(s)-\Delta_{\gamma_0}e^{(t-s)\Delta_{\gamma_0}}w(s).
\end{equation}
Here the last term is only integrable on $[0,t-\varepsilon]$ 
for $\varepsilon>0$, as it has a singularity at $s=t$; cf.\
Proposition~\ref{JJPazy'-prop}.  As a remedy, one can use the improper Bochner integral
\begin{align}  \label{JJdint-id}
  \JJdashint_0^t \Delta_{\gamma_0} e^{(t-s)\Delta_{\gamma_0}} w(s) \,ds 
 =\lim_{\varepsilon\to0}
  \int_0^{t-\varepsilon} \Delta_{\gamma_0} e^{(t-s)\Delta_{\gamma_0}} w(s) \,ds.
\end{align}

\begin{lemma}
For every $w \in H^1(Q)$ the limit \eqref{JJdint-id} exists in $L_2(\Omega)$ and
\begin{equation}
  \label{JJeq:w}
  w(t)=e^{t\Delta_{\gamma_0}}w(0)
   +\int_0^t e^{(t-s)\Delta_{\gamma_0}}\partial_sw(s)\,ds
    -\JJdashint_0^t\Delta_{\gamma_0}e^{(t-s)\Delta_{\gamma_0}}w(s))\,ds.
\end{equation}
\end{lemma}

\newpage
\begin{proof}
When applied to \eqref{JJw-id}, the Fundamental Theorem for vector functions gives
\begin{align}
  [e^{(t-s) \Delta_{\gamma_0}} w(s)]_{s=0}^{s=t-\varepsilon} 
  = \int_0^{t-\varepsilon} (-\Delta_{\gamma_0}) e^{(t-s)\Delta_{\gamma_0}} w \,ds 
  + \int_0^{t-\varepsilon} e^{(t-s)\Delta_{\gamma_0}} \partial_s w \,ds.
\end{align}
The left-hand side converges,
for $e^{t \Delta_{\gamma_0}}$ is of type $(M,0)$ and the proof above gave $w\in C([0,T],L_2(\Omega))$,
so by bilinearity $ e^{(t-s)) \JJlap_{\gamma_0}} w(s) \rightarrow w(t)$ in $L_2(\Omega)$ for $s\to t^-$.

Moreover, by dominated convergence the last integral converges in $L_2(\Omega)$ for $\varepsilon\to 0^+$,
whence 
$\int_0^{t-\varepsilon} \Delta_{\gamma_0} e^{(t-s)\Delta_{\gamma_0}} w(s) \,ds$ does so. Then \eqref{JJeq:w} results.
\end{proof}

By substituting \eqref{JJeq:w} for $w(t)$ in \eqref{JJeq:uw},
one obtains, as terms with $\partial_s w$ cancel, 
\begin{equation}  \label{JJeq:D}
  u(t)= e^{t\JJlap_{\gamma_0}}u_0 +\int_0^t e^{(t-s)\JJlap}f\,ds
  +\int_0^t e^{(t-s)\JJlap}\JJlap w\,ds
  -\JJdashint_0^t\JJlap_{\gamma_0}e^{(t-s)\JJlap_{\gamma_0}}w\,ds. 
\end{equation}
While the last two integrals look highly similar, a further reduction 
is possible since $\JJlap=\JJlap_{\gamma_0}\JJlap_{\gamma_0}^{-1}\JJlap$ holds on $H^1(\Omega)$.
Hence they combine into a single improper integral,
\begin{equation}
   \label{JJeq:wdiff}
   -\JJdashint_0^t \JJlap_{\gamma_0} e^{(t-s)\JJlap_{\gamma_0}}(I-\JJlap_{\gamma_0}^{-1}\JJlap) w(s)\,ds.
\end{equation}
Because of \eqref{JJeq:KgQ} we have $(I-\JJlap_{\gamma_0}^{-1}\JJlap)w={\cal Q} w=K_0\gamma_0 w= K_0 g$ as $\gamma_0 w=g$,
and when this is applied via \eqref{JJeq:wdiff} in \eqref{JJeq:D}, 
we finally arrive at the desired solution formula:

\begin{proposition} \label{JJug-prop}
If $u\in X_1$ denotes the unique solution to the initial boundary value problem \eqref{JJeq:heat_ivp}
provided by Proposition~\ref{JJheat-prop}, then $u$ fulfils the identity 
\begin{equation} \label{JJeq:ug}
  u(t)=e^{t\Delta_{\gamma_0}}u_0+\int_0^t e^{(t-s)\JJlap}f(s)\,ds
  - \JJdashint_0^t \JJlap_{\gamma_0} e^{(t-s)\JJlap_{\gamma_0}}K_0 g(s)\,ds,
\end{equation}
where the improper Bochner integral converges in $L_2(\Omega)$ for every $t\in [0,T]$.  
\end{proposition}

\begin{remark}
  Despite the classical context, \eqref{JJeq:ug} seemingly first appeared in \cite{JJChJo18ax}.
\end{remark}

For $t=T$ the second term in \eqref{JJeq:ug} gives back $y_f=\int_0^T e^{(T-s)\JJlap}f(s)\,ds$,
cf.\ \eqref{JJyf-eq}.
But the \emph{full} influence of the boundary data $g$ on $u(T)$ is contained in the third term,
\begin{align} \label{JJzg-id}
  z_g = \JJdashint_0^T \JJlap e^{(T-s) \JJlap_{\gamma_0}} K_0 g(s) \,ds.
\end{align}
Even the basic fact that $g \mapsto z_g$ is a well-defined map is a non-trivial result;
it results at once for $t=T$ in Proposition~\ref{JJug-prop}. 
The map is clearly linear by the calculus of limits. 

In case $f=0$, $u_0=0$ it is seen from \eqref{JJeq:ug} that
$z_g = -u(T)$, so the estimate in Proposition~\ref{JJheat-prop} yields that 
$\|z_g\|_{L_2(\Omega)}\leq \sup_t\|u(t)\|_{L_2(\Omega)}\le c \|g\|_{H^{1/2}(S)}$.
This proves

\begin{lemma} \label{JJzg-lem}
  The linear operator $g\mapsto z_g$ is bounded $H^{1/2}(S)\to L_2(\Omega)$.
\end{lemma}

Moreover, Proposition~\ref{JJug-prop} gives for an arbitrary solution in $X_1$ of the heat
equation $u'-\Delta u=f$ with $\gamma_0 u=g$ on $S$ that there is a bijection $u(0)\leftrightarrow
u(T)$ given by
\begin{equation} \label{JJuTu0-id}
  u(T)=e^{T\Delta_{\gamma_0}}u(0)+y_f-z_g.  
\end{equation}
Indeed, the above breaks down as application of the bijection $e^{T\JJlap_{\gamma_0}}$, cf.\
Proposition~\ref{JJinj-prop} followed by translation in $L_2(\Omega)$  by the fixed vector $y_f-z_g$. 

The above considerations suffice for a proof of unique solvability in $X_1$ of the fully inhomogeneous final
value problem \eqref{JJeq:heat_fvp} for suitable data $(f,g,u_T)$.
Both the result and the proof are highly similar to the abstract Theorem~\ref{JJintro-thm},
but one should note, of course, the new clarification that the boundary data $g$ 
do appear in the compatibility condition, and only via $z_g$:

\begin{theorem}  \label{JJyz-thm}
For given data $(f,g,u_T)$ in $L_2(0,T; H^{-1}(\Omega))\oplus H^{1/2}(S)\oplus L_2(\Omega)$,
the final value problem \eqref{JJeq:heat_fvp} is solved by a function $u\in X_1$, whereby 
\begin{equation} 
   X_1= L_2(0,T; H^1(\Omega)) \bigcap C([0,T];L_2(\Omega)) \bigcap H^1(0,T; H^{-1}(\Omega)),
\end{equation}
if and only if the data in terms of \eqref{JJyf-eq} and \eqref{JJzg-id} satisfy the compatibility
condition 
\begin{equation}  \label{JJyz-cnd}
  u_T - y_f + z_g \in D(e^{-T \JJlap_{\gamma_0}}).
\end{equation}
In the affirmative case, $u$ is uniquely determined in $X_1$ and has the representation
\begin{align}\label{JJyz-id}
  u(t) = 
   e^{t \Delta_{\gamma_0}} e^{-T \JJlap_{\gamma_0}} (u_T - y_f + z_g) 
   + \int_0^t e^{(t-s)\JJlap} f \,ds - \JJdashint_0^t \JJlap_{\gamma_0} e^{(t-s) \JJlap_{\gamma_0}} K_0 \,ds,
\end{align}
where the three terms all belong to $X_1$ as functions of $t\in[0,T]$.
\end{theorem}

\begin{proof}
Given a solution $u\in X_1$, the bijection \eqref{JJuTu0-id} yields
$u_T=e^{T\JJlap_{\gamma_0}}u(0)+y_f-z_g$, so that \eqref{JJyz-cnd} necessarily holds. 
Inserting its inversion
$u(0)=e^{-T\JJlap_{\gamma_0}}(u_T-y_f+z_g)$ into the solution formula from
Proposition~\ref{JJug-prop} yields \eqref{JJyz-id}; thence uniqueness of $u$.

If \eqref{JJyz-cnd} does hold, $u_0=e^{-T\JJlap_{\gamma_0}}(u_T-y_f+z_g)$ is a vector in
$L_2(\Omega)$, so the initial value problem with data $(f,g,u_0)$ can be solved using
Proposition~\ref{JJheat-prop}. This yields a $u\in X_1$ that also solves 
\eqref{JJeq:heat_fvp}, since $u(T)=u_T$ holds by \eqref{JJuTu0-id} and the choice of $u_0$. 

The final regularity statement follows from the fact that $X_1$ also is the solution space for the
Cauchy problem in Proposition~\ref{JJheat-prop}: the improper integral in \eqref{JJyz-id} is a
solution in $X_1$ to \eqref{JJeq:heat_ivp} for data $(f,g,u_0)=(0,g,0)$, according to
Proposition~\ref{JJug-prop};  
the integral containing $f$ solves \eqref{JJeq:heat_ivp} for data $(f,0,0)$, hence is in $X_1$; the first term in \eqref{JJyz-id} solves
\eqref{JJeq:heat_ivp} for data $(0,0,e^{-T \JJlap_{\gamma_0}}v)$ for the vector $v=u_T - y_f + z_g$.
\end{proof}

Exploiting the above theorem, we let $Y_1$ stand for the set of admissible data corresponding to $X_1$. 
Within the broader space $L_2(0,T;H^{-1}(\Omega))\oplus H^{1/2}(\Gamma) \oplus L_2(\Omega)$, the
data space $Y_1$ is the subspace given
via the map $\Phi_1(f,g,u_T)= u_T-y_f+z_g$ as
\begin{align} \label{JJY1-id}
  Y_1= \Big\{(f,g,u_T) \Bigm| u_T - y_f + z_g \in D(e^{-T \Delta_{\gamma_0}}) \Big\}
     =D(e^{-T \Delta_{\gamma_0}} \Phi_1).
\end{align}
Naturally $Y_1$ is endowed with the graph norm of $e^{-T \JJlap_{\gamma_0}} \Phi_1$, 
that is, of the composite map $(f,g,u_T)\mapsto e^{-T \JJlap_{\gamma_0}}(u_T-y_f+z_g)$.
Taking $t=0$ in \eqref{JJyz-id}, it is clear that $e^{-T\JJlap_{\gamma_0}}\Phi_1(f,g,u_T)$ is the initial state $u(0)$ steered by $f$, $g$ 
to the final one $u(T)=u_T$.

In some details, the above-mentioned graph norm is given by
\begin{multline}  \label{JJY1-nrm}
  \| (f,g,u_T)\|_{Y_1}^2 =  \|u_T\|^2_{L_2(\Omega)} + \|g\|^2w_{H^{1/2}(Q)} +
  \|f\|^2_{L_2(0,T;H^{-1}(\Omega))}
\\
   + \int_\Omega\Big|e^{-T\JJlap_{\gamma_0}}\Big(u_T - \int_0^T\!e^{(T-s)\JJlap}f\,ds + 
   \JJdashint_0^{T} \JJlap_{\gamma_0} e^{(T-s)\JJlap_{\gamma_0}} K_0 g \,ds\Big)\Big|^2\,dx.
\end{multline} 
Hereby $\|e^{-T \JJlap_{\gamma_0}}(u_T-y_f+z_g)\|_{L_2(\Omega)}^2$ 
is written with explicit integrals to emphasize the
complexity of the fully inhomogeneous boundary and final value problem \eqref{JJeq:heat_fvp}. 

Completeness of $Y_1$ follows from continuity of $\Phi_1$, cf.\  Lemma~\ref{JJzg-lem} concerning
$z_g$. Indeed, its composition to the left with the closed operator $e^{-T\Delta_{\gamma_0}}$ in $L_2(\Omega)$ (cf.\
Proposition~\ref{JJinverse-prop}) is also closed. Therefore its domain
$D(e^{-T\Delta_{\gamma_0}} \Phi_1)=Y_1$
is complete with respect to the graph norm in \eqref{JJY1-nrm}. This is
induced by an inner product when $H^{-1}(\Omega)$ is given the equivalent norm 
$\JJvvvert f\JJvvvert_* =s(\JJlap_{\gamma_0}^{-1}f,\JJlap_{\gamma_0}^{-1}f)^{1/2}$, and when $H^{1/2}(Q)$ is normed as in Remark~\ref{JJHsS-rem}.
Hence $Y_1$ is a Hilbertable space.  

These preparations concerning the data space $Y_1$ lead to the following stability result:

\begin{theorem} \label{JJyz-thm'}
The unique solution $u$ of problem \eqref{JJeq:heat_fvp} lying in the Banach space $X_1$ 
depends continuously on the data  $(f,g,u_T)$ in the Hilbert space $Y_1$, when these are given 
 the norms in \eqref{JJXheat-nrm} and \eqref{JJY1-nrm}, respectively.
\end{theorem}
\begin{proof}
Boundedness of the solution operator 
$(f,g,u_T)\mapsto u$ is seen by inserting the
expression  $u_0 = e^{-T \Delta_{\gamma_0}} (u_T - y_f +z_g)$ from  \eqref{JJuTu0-id} into the
estimate in Proposition~\ref{JJheat-prop},
\begin{align}  
  \|u\|_{X_1}^2 \leq 
  c ( \|e^{-T \Delta_{\gamma_0}} (u_T - y_f +z_g)\|_{L_2(\Omega)}^2 + \|f\|_{L_2(0,T;H^{-1}(\Omega))}^2 + \|g\|_{H^{1/2}(S)}^2 ).
\end{align}  
Adding $\|u_T\|_{L_2(\Omega)}^2$ on the right-hand side, one arrives at $\|(f,g,u_T)\|_{Y_1}^2$.
\end{proof}

Taken together, Theorem~\ref{JJyz-thm} and Theorem~\ref{JJyz-thm'} yield the final result:

\begin{theorem} \label{JJyz-thm''}
When $\Omega$ is a $C^\infty$ smooth, open bounded set in $\JJRn$ for
$n\ge2$, and $T>0$, then
the fully inhomogeneous final value heat conduction problem 
\begin{equation}
  \left\{
  \begin{aligned}
  \partial_t u\JJmlap u&=f\quad\text{ in $\,]0,T[\,\times \Omega$},
  \\
   \gamma_0 u&=g \quad\text{ on $\,]0,T[\,\times \Gamma$},
\\
    u(T)&=u_T \quad\text{ in $\Omega$},    
  \end{aligned}
  \right.
\end{equation}
is well posed with solutions $u$ and data $(f,g,u_T)$ belonging to
the spaces $X_1$ and $Y_1$ defined in \eqref{JJX1-id}, \eqref{JJY1-id}, and normed as in \eqref{JJXheat-nrm}, \eqref{JJY1-nrm}, respectively.  
\end{theorem}

\end{document}